\documentclass[12pt,table]{amsart}
\usepackage{amsmath}
\usepackage[margin=1.0in]{geometry}
\usepackage{amscd,amsmath,amsxtra,amsthm,amssymb,stmaryrd,xr,mathrsfs,mathtools,enumerate,commath, comment, enumitem, graphicx}
\usepackage{amsfonts, amssymb, amsthm, amsmath, braket, hyperref, mathtools, comment, mathrsfs, enumitem, cleveref}
\usepackage{stmaryrd}
\usepackage{amsfonts}
\usepackage{orcidlink}
\usepackage{multirow}
\usepackage{xcolor}
\usepackage{commath}
\usepackage{float}
\usepackage{comment}
\usepackage{tikz-cd}
\usepackage{tkz-graph}
\usepackage{colonequals}
\usepackage{hyperref}
\usepackage{graphicx}
\usepackage{bigstrut}
\usepackage{tabularx}
\usepackage{longtable, multirow, array}

\allowdisplaybreaks
\usepackage[T1]{fontenc}       
\usepackage{lmodern}

\usepackage{longtable} 
\usepackage{pdflscape} 
\usepackage{booktabs}
\usepackage{hyperref}

\definecolor{darkteal}{RGB}{0, 96, 96} 

\usepackage{xcolor}
\usepackage{hyperref}

\usepackage{xcolor}
\definecolor{academicteal}{RGB}{0, 47, 47}

\usepackage{hyperref}
\hypersetup{
    colorlinks=true,       
    linkcolor=blue,       
    citecolor=red, 
    filecolor=magenta,
    urlcolor=academicteal  
}

\usepackage[all,cmtip]{xy}
\DeclareFontFamily{U}{wncy}{}
\DeclareFontShape{U}{wncy}{m}{n}{<->wncyr10}{}
\DeclareSymbolFont{mcy}{U}{wncy}{m}{n}
\DeclareMathSymbol{\Sh}{\mathord}{mcy}{"58}
\usepackage[T2A,T1]{fontenc}
\usepackage[OT2,T1]{fontenc}

\usepackage{tikz}
\usetikzlibrary{shapes.geometric}
\usetikzlibrary{decorations.markings}
\tikzset{every loop/.style={min distance=10mm,looseness=10}}
\tikzstyle{vertex}=[auto=left,circle,minimum size=1pt,inner sep=0pt]

\newtheorem{theorem}{Theorem}[section]
\newtheorem{lemma}[theorem]{Lemma}

\newtheorem*{theorem*}{Theorem}
\newtheorem*{ass*}{Assumption}

\newtheorem{remark}{Remark}

\numberwithin{equation}{section}

\begin{document}
\allowdisplaybreaks

\title[]{Orthogonality relation between $GL(3)$ coefficients and additive characters with linear and non-linear phases}

\author[Venkatasubbareddy]{K. Venkatasubbareddy\, \orcidlink{0009-0006-7987-0835}}
\address[Venkatasubbareddy]{Theoretical Statistics and Mathematics Unit, ISI Kolkata, Kolkata, India-700108}
\email{venkatasubbareddy7313@gmail.com}


\keywords{Vaughan's identity, exponential sums,  Fourier coefficients, Hecke-Maass forms}
\subjclass[2020]{Primary: 11F30; Secondary: 11F66, 11L07, 11L20}

\maketitle

\maketitle
\begin{abstract}
We study the distribution and oscillatory properties of prime-supported Fourier coefficients attached to a Hecke--Maass form for $\mathrm{SL}(3, \mathbb{Z})$ twisted by additive characters. For linear phases involving irrationals of finite type, we construct an automorphic variant of Vaughan's identity and exploit bilinear cancellation to obtain an unconditional power-saving estimate. For non-linear phases $e(\alpha n^\theta)$ with $0 < \theta < 1$, we refine the Type~II bilinear forms via Weyl differencing and modern exponent pair methods. Combined with the analysis of untwisted $\mathrm{GL}(3)$ sums, this framework breaks the $2/3$-barrier of Jiang and L\"u and widens the admissible range of non-trivial cancellation.
\end{abstract}


\section{Introduction}

Estimating exponential sums twisted by arithmetic functions supported on primes is a central theme in analytic number theory \cite{JiangLu2017}. For a general arithmetic sequence $\mathcal{A} = \{a(n)\}_{n \geq 1}$, one seeks non-trivial upper bounds for sums of the shape
\begin{equation*}
S(\mathcal{A}, X, \eta) := \sum_{n \leq X} \Lambda(n) a(n) e(\alpha n^\eta) \quad (\alpha \in \mathbb{R},\ \eta > 0),
\end{equation*}
where $\Lambda(n)$ is the classical von Mangoldt function and $e(x) := \exp(2\pi i x)$. In the benchmark setting where $a(n) \equiv 1$, I.~M.~Vinogradov established his celebrated estimates for the linear phase $\eta = 1$ in connection with the ternary Goldbach problem, alongside square-root power savings for the non-linear phase $\eta = 1/2$. Later refinements based on combinatorial identities by Iwaniec and Kowalski \cite{IwaniecKowalski} and Heath-Brown's identity by Ren \cite{LiuRen2012} extended non-trivial cancellation across the general non-linear spectrum $0 < \eta < 1$.

When the arithmetic weights $a(n)$ arise from the Fourier coefficients of automorphic forms, the problem probes the deep orthogonality between automorphic representations, prime numbers, and additive characters \cite{Fouvry-Ganguly2014}, \cite{Hou-Jiang-Lu}, \cite{JiangLu2019}.


For higher-rank groups such as $\mathrm{SL}(m, \mathbb{Z})$ with $m \geq 3$, investigating these sums becomes substantially more difficult (see \cite{JiangLu2017}). In this paper, we focus on the case $m = 3$. Let $F$ be a Hecke--Maass form for $\mathrm{SL}(3, \mathbb{Z})$ of type $\nu = (\nu_1, \nu_2)$ with normalized Fourier coefficients $A_F(n_1, n_2)$ satisfying $A_F(1, 1) = 1$. The Godement--Jacquet $L$-function attached to $F$ is defined for $\Re s > 1$ by the Dirichlet series and Euler product
\begin{align*}
L(s, F) &= \sum_{n=1}^\infty \frac{A_F(n, 1)}{n^s} = \prod_p \prod_{j=1}^3 \left( 1 - \frac{\alpha_F(p, j)}{p^s} \right)^{-1} \notag \\
&= \prod_p \left( 1 - \frac{A_F(p, 1)}{p^s} + \frac{A_F(1, p)}{p^{2s}} - \frac{1}{p^{3s}} \right)^{-1},
\end{align*}
where the local Satake parameters satisfy
\begin{equation*}
\sum_{j=1}^3 \alpha_F(p, j) = A_F(p, 1), \quad \sum_{1 \leq j < k \leq 3} \alpha_F(p, j)\alpha_F(p, k) = A_F(1, p),  \prod_{j=1}^3 \alpha_F(p, j) = 1,
\end{equation*}
and the Fourier coefficients satisfy $A_F(p,1)=\overline{A_F(1,p)}$ and the multiplicative relation
\begin{equation*}
    A_F(m,1)A_F(n,1)=\sum_{d|(m,n)}A_F\left(\frac{mn}{d^2},d\right). 
\end{equation*}
While the generalized Ramanujan conjecture asserts that $|\alpha_F(p, j)| = 1$, the best known unconditional bounds toward Ramanujan on $\mathrm{GL}(3)$ due to Kim and Sarnak \cite{KimSarnak2003} yield
\begin{equation*}
|\alpha_F(p, j)| \leq p^{5/14}, \quad \text{whence} \quad A_F(n, 1) \ll n^{5/14} d_3(n).
\end{equation*}
The Rankin--Selberg convolution $L$-function associated with $F$ is defined for $\Re s > 1$ by the Euler product
\begin{equation*}
L(s, F \times F) = \prod_p \prod_{j=1}^3 \prod_{k=1}^3 \left(1 - \frac{\alpha_F(p, j)\overline{\alpha_F(p, k)}}{p^s}\right)^{-1}.
\end{equation*}
The coefficient of $p^{-s}$ in the Euler product expansion of $L(s, F \times F)$ is $|A_F(p, 1)|^2$.
By the Rankin--Selberg theory, 
the $L$-function $L(s, F \times F)$ has a meromorphic continuation to the entire complex plane with a simple pole at $s = 1$. Applying Perron's formula and contour shift to the Dirichlet series representation of $L(s, F \times F)$, the residue at the simple pole $s = 1$ yields the mean-square estimate (see \cite[Remark 12.1.8]{Goldfeld2006})
\begin{equation*}
\sum_{n \leq N} |A_F(n, 1)|^2 \ll_F N.
\end{equation*}

Taking the logarithmic derivative of $L(s, F)$ defines the generalized von Mangoldt function $\Lambda_F(n)$
\begin{equation*}
-\frac{L'}{L}(s, F) = \sum_{n=1}^\infty \frac{\Lambda_F(n)}{n^s} = \sum_{n=1}^\infty \frac{\Lambda(n) a_F(n)}{n^s}, \quad a_F(p^k) = \sum_{j=1}^m \alpha_F(p, j)^k.
\end{equation*}
Consequently, 
\begin{equation*}
\Lambda_F(n) = 
\begin{cases}
\log p \sum_{j=1}^3 \alpha_F(p, j)^k & \text{if } n = p^k, \\
0 & \text{otherwise},
\end{cases}
\end{equation*}

Similarly, the generalized M\"obius function $\mu_F(n)$ is generated by the reciprocal of the Godement--Jacquet $L$-function $L(s, F)$ for $\Re s > 1$
\begin{equation*}
L^{-1}(s, F) = \sum_{n=1}^\infty \frac{\mu_F(n)}{n^s} = \prod_p \prod_{j=1}^3 \left( 1 - \frac{\alpha_F(p, j)}{p^s} \right) = \prod_p \left( 1 - \frac{A_F(p, 1)}{p^s} + \frac{A_F(1, p)}{p^{2s}} - \frac{1}{p^{3s}} \right) .
\end{equation*}
By comparing coefficients on both sides of the Euler product, we see that $\mu_F(n)$ is a multiplicative function supported exclusively on cubefree (3-power-free) integers, vanishing whenever $p^4 \mid n$ for some prime $p$. At prime powers, its values are given explicitly in terms of the Satake roots and Fourier coefficients by
\begin{equation*}
\mu_F(n) = 
\begin{cases}
1 & \text{if } n = 1, \\
-A_F(p, 1) & \text{if } n = p, \\
A_F(1, p) & \text{if } n = p^2, \\
-1 & \text{if } n = p^3, \\
0 & \text{if } p^4 \mid n \text{ for some prime } p.
\end{cases}
\end{equation*}
In contrast to the classical M\"obius function $\mu(n)$, the coefficients $\mu_F(n)$ are unbounded and are not supported solely on square-free numbers.

In this paper, we study the twisted exponential sum
\begin{equation*}
S(F, X, \eta) := \sum_{n \leq X} \Lambda_F(n) e(\alpha n^\eta) \quad (\alpha \in \mathbb{R},\ 0 < \eta \leq 1)
\end{equation*}
for a Hecke--Maass form $F$ on $\mathrm{SL}(3, \mathbb{Z})$.

\begin{remark}
For any $0\leq \eta\leq 1$, we can write
\begin{align*}
S(F,X, \eta):&=\sum_{n \leq X} \Lambda_F(n) e(n^\eta\alpha) \\
&= \sum_{p \leq X} \log p \sum_{j=1}^3 \alpha_F(p, j) e(p^\eta\alpha) + \sum_{\substack{n = p^k \leq X \\ k \geq 2}} \log p \sum_{j=1}^3 \alpha_F(p, j)^k e(n^\eta\alpha) \\
&= \sum_{n \leq X} \Lambda(n) A_F(n, 1) e(n^\eta\alpha) + O\Biggl( \sum_{\substack{n = p^k \leq X \\ k \geq 2}} \log p\,\Bigl( \Bigl| A_F(p^k, 1) \Bigr| +\Bigl|\sum_{j=1}^3 \alpha_F(p, j)^k \Bigr|\Bigr) \Biggr).
\end{align*}
For the higher prime powers contribution, we have
\begin{equation*}
A_F(p^2, 1) = A_F(p, 1)^2 - A_F(1, p), \quad a_F(p^2) = A_F(p, 1)^2 - 2 A_F(1, p).
\end{equation*}
Hence,
\begin{align*}
\sum_{\substack{p^k \leq X \\ k \geq 2}} \bigl( |A_F(p^k, 1)| + |a_F(p^k)| \bigr) 
\ll & \sum_{\substack{p^2 \leq X }} |A_F(p, 1)|^2 + \sum_{\substack{p^2 \leq X }} |A_F(1, p)| \\
&\qquad+\sum_{\substack{p^k \leq X \\ k \geq 3}} |A_F(p^k, 1)| + \sum_{\substack{p^k \leq X \\ k \geq 3}} |a_F(p^k)|\\
\ll&\sum_{\substack{p \leq \sqrt{X} }} |A_F(p, 1)|^2 + \Bigl(X^{\frac{1}{2}}\sum_{\substack{p \leq \sqrt{X} }} |A_F(1, p)|^2\Bigr)^{\frac{1}{2}} +\sum_{p \leq X^{1/3} } p^{\frac{15}{14}+\varepsilon}\\
\ll &x^{1/2+\varepsilon}+x^{29/42+\varepsilon}.
\end{align*}
Therefore, 
\begin{align*}
\sum_{n \leq X} \Lambda_F(n) e(n^\eta\alpha) = \sum_{n \leq X} \Lambda(n) A_F(n, 1) e(n^\eta\alpha) + O\bigl(X^{29/42+\varepsilon}\bigr).
\end{align*}
\end{remark}

Since the secondary error term is strictly smaller than $X^{3/4}$, non-trivial estimates for $S(F, X, \eta)$ directly translate into cancellation for the prime sum $\sum_{n \leq X} \Lambda(n) A_F(n, 1) e(\alpha n^\eta)$. We divide the analysis into the linear phase regime $\eta = 1$ and the non-linear phase regime $\eta = \theta \in (0, 1)$.

\noindent
\textbf{The Linear Phase Regime.}
When $\eta = 1$, we evaluate the orthogonality of $\Lambda_F(n)$ against the classical linear phase $e(n\alpha)$. For $\mathrm{GL}(2)$ forms, Fouvry and Ganguly \cite{Fouvry-Ganguly2014} proved the strong orthogonality estimate 
$$\sum_{n \leq X} \Lambda(n) A_F(n) e(n\alpha) \ll X \exp(-c\sqrt{\log X}).$$
For $\mathrm{GL}(3)$, Hou, Jiang, and L\"u \cite{Hou-Jiang-Lu} obtained the  analogous sub-exponential bound 
$$\sum_{n \leq X} \Lambda(n) A_F(n, 1) e(n\alpha) \ll X \exp(-c\sqrt{\log X}).$$
While these results establish genuine asymptotic orthogonality, the rate of decay is strictly sub-polynomial and falls short of a power saving.

A major objective of this work is to demonstrate that by restricting $\alpha$ to irrationals of type $1$, one can cross this barrier and extract a substantial, unconditional \textbf{power-saving estimate} of the form $O(X^{1-\delta+ \varepsilon})$ with an explicit $\delta > 0$. An irrational number $\alpha$ is said to be of \textit{type $t\geq 1$}, if $t$ is the supremum of all $\delta$ for which 
\begin{equation*}
    \liminf_{q\rightarrow \infty}q^\delta||q\alpha||=0, 
\end{equation*}
where $\displaystyle||\alpha||=\min_{n\in \mathbb{Z}}|\alpha-n|$. From Roth's theorem, it follows that for any algebraic irrational number $\alpha$, we have $|q\alpha-a|\gg q^{-(1+\varepsilon)}$. Therefore, all algebraic irrationals are of type 1. Moreover, it is well known that almost all irrationals are of type $1$.

For untwisted $\mathrm{GL}(3)$ Fourier coefficients, S.~D.~Miller \cite{Miller2006} established the additively twisted estimate
\begin{equation*}
\sum_{n \leq X} A_F(n, 1) e(n\alpha) \ll_{F,  \varepsilon} X^{\frac{1}{2} + \rho +  \varepsilon},
\end{equation*}
which holds unconditionally with $\rho = 1/4$ (and conjecturally with $\rho = 0$). By constructing an analogue of Vaughan's identity for the Godement--Jacquet $L$-function $L(s, F)$ and coupling Miller's bound with bilinear reciprocity estimates over modular intervals, we exploit genuine cancellation in both Type~I and Type~II regimes. This yields an unconditional power-saving reduction of $X^{1/20}$.

\begin{theorem}\label{thm:linear}
Let $F$ be a Hecke--Maass form for $\mathrm{SL}(3, \mathbb{Z})$ and $\alpha$ an irrational number of type $1$. Then for any $ \varepsilon > 0$, we have
\begin{equation*}
S(F, X) = \sum_{n \leq X} \Lambda_F(n) e(n\alpha) \ll_{F,  \varepsilon} X^{\frac{9+2\rho}{10}+ \varepsilon},
\end{equation*}
where $\rho$ denotes the parameter in Miller's bound. In particular, substituting the unconditional value $\rho = 1/4$ yields
\begin{equation*}
S(F, X) \ll_{F,  \varepsilon} X^{\frac{19}{20}+ \varepsilon}.
\end{equation*}
\end{theorem}

\noindent
\textbf{The Non-Linear Phase Regime .} When the character has a non-linear phase $e(\alpha n^\theta)$ with $0 < \theta < 1$, bounding the sum becomes significantly more delicate.

To motivate our approach, we first consider the untwisted sum of $\mathrm{GL}(3)$ Fourier coefficients
\begin{equation*}
\sum_{n \leq X} \lambda_F(n, 1) e(\alpha n^\theta).
\end{equation*}
For $\theta \neq 1/3$, Ren and Ye \cite{Ren-Ye} established the bound $O\big(X^{\frac{3\theta}{2}} \log X\big)$ via the $\mathrm{GL}(3)$ Voronoi summation formula. For the range $1/3 < \theta < 7/9$, Kumar, Mallesham, and Singh \cite{Sumit-Mallesham-saurabh} applied Munshi's $\delta$-method to prove the bound $O\big(X^{\frac{3}{4} + \frac{9\theta}{28} +  \varepsilon}\big)$. Since the Ren--Ye estimate satisfies $X^{\frac{3\theta}{2}} \leq X^{\frac{3}{4} + \frac{9\theta}{28}}$ for all $\theta \leq 7/11 \approx 0.636$, combining the two results yields the estimate
\begin{equation*}
\sum_{n \leq X} A_F(n, 1) e(\alpha n^\theta) \ll_{F, \alpha} X^{\frac{3}{4} + \frac{9\theta}{28} +  \varepsilon}
\end{equation*}
across the entire range $0 < \theta < 7/9$, $\theta\neq 1/3$.

For prime-supported sums, Jiang and L\"u \cite{JiangLu2017} obtained the non-trivial bound

\begin{equation*}
    S(F,X,\theta)\ll (X^{5/6}+X^{(4+3\theta)/6}+X^{1-\theta/2})\log^{9/2}X
\end{equation*}
only up to $\theta < 2/3$. Their limitation arises from applying generic bilinear inequalities to the Type~II sums.

In this paper, we bypass the $2/3$-barrier by refining the bilinear estimates. Applying Cauchy--Schwarz followed by Weyl differencing reduces the Type~II sum to a one-dimensional exponential sum. Evaluating this inner sum using Bourgain's exponent pair $(13/84, 55/84)$ yields substantially stronger cancellation, establishing the following result.

\begin{theorem}\label{thm:non-linear}
Let $F$ be a Hecke--Maass form for $\mathrm{SL}(3, \mathbb{Z})$, $\alpha \in \mathbb{R}$, and $0 < \theta < \frac{497}{691} \approx 0.71925$, $\theta\neq \frac{1}{3}$. Then for any $ \varepsilon > 0$,
\begin{equation*}
S(F, X, \theta) = \sum_{n \leq X} \Lambda_F(n) e(\alpha n^\theta) \ll_{\alpha, F,  \varepsilon} X^{\frac{3143+443\theta}{3528}+ \varepsilon}.
\end{equation*}
\end{theorem}

\remark{
Theorem~\ref{thm:non-linear} broadens the admissible range of non-linearity from $\theta < 2/3 \approx 0.66667$ up to $\theta < 497/691 \approx 0.71925$. Furthermore, comparing the balanced exponent $\frac{3143+443\theta}{3528}$ against $\max\{5/6, 1-\theta/2, (4+3\theta)/6\}$ shows that our bound improves Jiang and L\"u's estimate on the intervals $\theta \in (0, 385/2207) \cup (791/1321, 497/691)$. In Section~ \ref{Sec-analysis-exp-pair}, we provide a parametric analysis of how the choice of exponent pair $(k, \lambda)$ governs the admissible range of $\theta$, demonstrating that decreasing $k$ systematically pushes the reach of the method toward the critical decomposition barrier $\theta < 7/9 \approx 0.77778$.

}

\subsection{An Analogue of Vaughan's Identity}

We have the Dirichlet series identity 
\begin{align*}
-\frac{L'(s, F)}{L(s, F)} 
&= H_1(s) - L'(s, F) H_2(s) - L(s, F) H_1(s) H_2(s) \notag \\
&\quad + \left(-\frac{L'(s, F)}{L(s, F)} - H_1(s)\right) \bigl(1 - L(s, F) H_2(s)\bigr),
\end{align*}
where
\begin{equation*}
H_1(s) = \sum_{n \leq U} \frac{\Lambda_F(n)}{n^s} \quad \text{and} \quad H_2(s) = \sum_{n \leq V} \frac{\mu_F(n)}{n^s},
\end{equation*}
for any $U, V \geq 1$.

Thus, for $n>UV$, we can write 
\begin{equation}
    \Lambda_F(n) = a_1(n) + a_2(n) + a_3(n) + a_4(n),\label{eq:split into Type I II sums}
\end{equation}
where
\begin{align*}
a_1(n) &= \begin{cases} \Lambda_F(n) & \text{if } n \leq U, \\ 0 & \text{otherwise}, \end{cases} \\
a_2(n) &=- \sum_{\substack{n = m l \\ m \leq V}} \mu_F(m) A_F(l, 1) \log l, \\
a_3(n) &=- \sum_{n = m l} \psi_F(m) A_F(l, 1), \quad \text{with} \quad \psi_F(m) = \sum_{\substack{m = m_1 m_2 \\ m_1 \leq U, \, m_2 \leq V}} \Lambda_F(m_1) \mu_F(m_2), \\
a_4(n) &= \sum_{\substack{n = m l \\ m > U, \, l > 1}} \Lambda_F(m) \varphi_F(l), \quad \text{with} \quad \varphi_F(l) = \sum_{\substack{l = l_1 l_2 \\ n_1 \leq V}} \mu_F(l_1) A_F(l_2, 1).
\end{align*}

Notice that if $l \leq V$, then $\varphi_F(l) = \sum_{l = l_1 l_2} \mu_F(l_1) A_F(l_2, 1)$, which is the coefficient of $l^{-s}$ in $\frac{1}{L(s, F)} L(s, F) = 1$. Therefore,
\begin{equation*}
\varphi_F(l) = 0 \quad \text{for } 2 \leq l \leq V.
\end{equation*}

\section{Linear case}
\subsection{Preliminaries}
\begin{lemma}\label{lem:linear_bilinear_1_4}
Let $M, N \geq 2$ and $X \asymp MN$. Let $\{a(m): m \sim M\}$ and $\{b(n): n \sim N\}$ be two arbitrary sequences of complex numbers satisfying the mean-value bounds
\begin{equation*}
\label{eq:l2_bounds}
\sum_{m \sim M} |a(m)|^2 \ll M \log^{c_0} M \qquad \text{and} \qquad \sum_{n \sim N} |b(n)|^2 \ll N \log^{c_0} N
\end{equation*}
for some constant $c_0 > 0$. Suppose $\alpha \in \mathbb{R}$ satisfies the Dirichlet rational approximation
\begin{equation*}
\left| \alpha - \frac{a}{q} \right| \leq \frac{1}{q^2}, \qquad (a, q) = 1, \qquad q \geq 1.
\end{equation*}
Then, we have 
\begin{equation}\label{eq:bilinear_1_4_bound}
S := \sum_{m \sim M} \sum_{n \sim N} a(m) b(n) e(\alpha mn) \ll X^{1+\varepsilon}\left( \frac{1}{M} + \frac{1}{N} + \frac{1}{q} + \frac{q}{X} \right)^{1/4}.
\end{equation}
\end{lemma}

\begin{proof}
Without loss of generality, we may assume $q \leq X$; otherwise, the right-hand side of \eqref{eq:bilinear_1_4_bound} exceeds the trivial bound $\ll X \log^{c_0} X$ coming from Cauchy--Schwarz.

Applying the Cauchy--Schwarz inequality across the $n$-variable to separate $b(n)$ gives
\begin{align*}
|S|^2 \leq &\left( \sum_{n \sim N} |b(n)|^2 \right) \sum_{n \sim N} \left| \sum_{m \sim M} a(m) e(\alpha mn) \right|^2 \\
\ll& N \log^{c_0} N \sum_{n \sim N} \sum_{m_1 \sim M} \sum_{m_2 \sim M} a(m_1) \overline{a(m_2)} e\left(\alpha (m_1 - m_2)n\right)\\
\ll &N \log^{c_0} N \sum_{m \sim M} |a(m)| \sum_{|h| < M} |a(m+h)| \left| \sum_{n \sim N} e(\alpha hn) \right|
\end{align*}

We now apply the Cauchy--Schwarz inequality to the $m$-summation
\begin{align*}
|S|^2 \ll&M^{1/2}N \log^{c_0}N \log^{c_0/2} \left(\sum_{m \sim M} \left( \sum_{|h| < M} |a(m+h)| \left| \sum_{n \sim N} e(\alpha hn) \right| \right)^2\right)^{1/2}.
\end{align*}
Applying Cauchy--Schwarz to the inner $h$-sum with the weights $|a(m+h)|$ and noting that $\sum_{|h| < M} |a(m+h)|^2 \ll M \log^{c_0} M$ uniformly in $m$, we get
\begin{align*}
|S|^2 \ll&M^{3/2}N \log^{c_0} N \log^{3c_0/2}  \left(\sum_{|h| < M} \left| \sum_{n \sim N} e(\alpha hn) \right|^2\right)^{1/2}.
\end{align*}

Because the phase function is purely linear in $n$, the inner sum is a geometric series and using the standard reciprocal sum estimate
\begin{equation*}
    \sum_{n\sim N}e(\alpha hn)\ll \min\left(N, \, \frac{1}{\|\alpha h\|}\right),
\end{equation*}
we get
\[
\sum_{|h| < M} \min\left(N, \, \frac{1}{\|\alpha h\|}\right) \ll \left(M + N + q + \frac{MN}{q}\right) \log(2q).
\]
Substituting this estimate, we derive
\[
|S| \ll (MN)^{3/4} \left(M + N + q + \frac{MN}{q}\right)^{1/4} \log^{c_1} MN\ll 
X^{1+\varepsilon} \left( \frac{1}{M} + \frac{1}{N} + \frac{1}{q} + \frac{q}{X} \right)^{1/4}.\]
\end{proof}

\begin{lemma}\label{lem:unified_rankin_selberg}
Let $F$ be a Hecke--Maass form for $\mathrm{SL}(3, \mathbb{Z})$, and let $A_F(n, 1)$, $\mu_F(n)$, and $\Lambda_F(n)$ denote the arithmetic coefficients defined as above. Then for any $x \geq 2$, there exists a positive constant $c > 0$ such that
\begin{align}
\sum_{n \leq x} |\mu_F(n)|^2 &= cx + O\left(x^{\frac{6}{7}+ \varepsilon}\right) , \label{eq:muF_mean_square} \\
\sum_{n \leq x} d(n) |\mu_F(n)|^2 &\ll x \log x, \label{eq:d_muF_mean_square} \\
\sum_{n \leq x} d(n) |A_F(n, 1)|^2 &\ll x \log x, \label{eq:d_AF_mean_square} \\
\sum_{n \leq x} d(n) |\Lambda_F(n)|^2 &\ll x \log^3 x, \label{eq:d_LambdaF_mean_square}
\end{align}
where the implied constants depend only on the form $F$ and $ \varepsilon > 0$.
\end{lemma}

\begin{proof}
Consider the Dirichlet series attached to $|\mu_F(n)|^2$ and $d(n)|\mu_F(n)|^2$
\begin{equation*}
D(\mu_F, s) := \sum_{n=1}^\infty \frac{|\mu_F(n)|^2}{n^s} \quad \text{and} \quad D_d(\mu_F, s) := \sum_{n=1}^\infty \frac{d(n)|\mu_F(n)|^2}{n^s}.
\end{equation*}
Since $\mu_F(n)$ is multiplicative and supported on cubefree integers ($p^4 \nmid n$), for $\Re s > 1$ we have the Euler products
\begin{align*}
D(\mu_F, s) &= \prod_p \left( 1 + \frac{|\mu_F(p)|^2}{p^s} + \frac{|\mu_F(p^2)|^2}{p^{2s}} + \frac{|\mu_F(p^3)|^2}{p^{3s}} \right) , \\
D_d(\mu_F, s) &= \prod_p \left( 1 + \frac{2|\mu_F(p)|^2}{p^s} + \frac{3|\mu_F(p^2)|^2}{p^{2s}} + \frac{4|\mu_F(p^3)|^2}{p^{3s}} \right).
\end{align*}
Using the values $\mu_F(p) = -A_F(p, 1)$, $\mu_F(p^2) = A_F(1, p)$, and $\mu_F(p^3) = -1$, we factorize these series by extracting the dominant factors from $L(s, F \times F)$ as
\begin{equation*}
D(\mu_F, s) = L(s, F \times F) U_1(s) \quad \text{and} \quad D_d(\mu_F, s) = L^2(s, F \times F) U_2(s) .
\end{equation*}
By the Kim--Sarnak bound $|\alpha_F(p, j)| \leq p^{5/14}$, the Euler products $U_1(s)$ and $U_2(s)$ converge absolutely in the half-plane $\Re s > 6/7$. 

Because $L(s, F \times F)$ has a simple pole at $s = 1$, $D(\mu_F, s)$ has a simple pole at $s = 1$, and standard contour integration yields the asymptotic formula \eqref{eq:muF_mean_square}. Meanwhile, $D_d(\mu_F, s)$ has a double pole at $s = 1$, which immediately yields the upper bound \eqref{eq:d_muF_mean_square}.

Similarly, for the Dirichlet series $$D_d(A_F, s) := \sum_{n=1}^\infty d(n)|A_F(n, 1)|^2 n^{-s},$$ the Euler product decomposition gives
\begin{equation*}
D_d(A_F, s) = L^2(s, F \times F) U_3(s),
\end{equation*}
where $U_3(s)$ converges absolutely in $\Re s > 6/7$. The pole of order $2$ at $s = 1$ leads directly to \eqref{eq:d_AF_mean_square}.

To establish \eqref{eq:d_LambdaF_mean_square}, we separate the sum into primes and higher prime powers as
\begin{align*}
\sum_{n \leq x} d(n) |\Lambda_F(n)|^2& = 2 \sum_{p \leq x} |A_F(p, 1)|^2 \log^2 p + \sum_{k \geq 2} \sum_{p^k \leq x} (k+1)|a_F(p^k)|^2 \log^2 p 
\end{align*}
Recall that $a_F(p^2)=A_F(p,1)^2-2A_F(1,p)$ and $a_F(p^3)=A_F(p,1)^3-3A_F(p,1)A_F(1,p)+3$. Using the Kim-Sarnak bound $A_F(n,1)\ll n^{5/14+\varepsilon}$, gives 
\begin{align*}
    \sum_{p\leq x^{1/2}}|a_F(p^2)|^2\log^2p&\ll x^\varepsilon \sum_{p\leq x^{1/2}}|A_F(p,1)|^4+x^\varepsilon\sum_{p\leq x^{1/2}}|A_F(1,p)|^2\ll x^{6/7+\varepsilon}, \\
    \sum_{p\leq x^{1/3}}|a_F(p^3)|^2\log^2p&\ll x^\varepsilon \sum_{p\leq x^{1/3}}|A_F(p,1)|^6+x^\varepsilon\sum_{p\leq x^{1/3}}|A_F(p,1)|^2|A_F(1,p)|^2+x^{1/3+\varepsilon} \\
    &\ll x^{17/21+\varepsilon}+x^{12/21+\varepsilon}+x^{1/3+\varepsilon}\\
    &\ll x^{17/21+\varepsilon},
\end{align*}
and 
\begin{align*}
    \sum_{k\geq 4}\sum_{p^k\leq x}|a_F(p^k)|^2\log^2p\ll \sum_{k\geq 4}\sum_{p^k\leq x}p^{5k/7+\varepsilon}\ll \sum_{k\geq 4}x^{\frac{5}{7}+\frac{1}{k}+\varepsilon}\ll x^{27/28+\varepsilon}.
\end{align*}
Since $\sum_{p\leq x}|A_F(p,1)|^2\log^2p\ll x\log^3x$ and $x^{\frac{27}{28}+ \varepsilon} = o(x \log^3 x)$, the main contribution yields \eqref{eq:d_LambdaF_mean_square}.
\end{proof}

\begin{lemma}\label{lem:mean-square-estimates}
Let $F$ be a Hecke--Maass form for $\mathrm{SL}(3, \mathbb{Z})$ with normalized Fourier coefficients $A_F(n, 1)$. Let $\Lambda_F(m)$, $\psi_F(m)$, and $\varphi_F(l)$ be the coefficients defined as above. Then for any $X \geq 2$, we have the uniform bounds
\begin{align}
\sum_{m \leq X} |\Lambda_F(m)|^2 &\ll X \log X, \label{eq:lem-LambdaF} \\
\sum_{m \leq X} |\psi_F(m)|^2 &\ll X \log^3 X, \label{eq:lem-aF} \\
\sum_{l \leq X} |\varphi_F(l)|^2 &\ll X \log^3 X . \label{eq:lem-bF}
\end{align}
\end{lemma}

\begin{proof}
Since $\Lambda_F(m)$ is supported only on prime powers, we split the sum as
\begin{equation*}
\sum_{m \leq X} |\Lambda_F(m)|^2 = \sum_{p \leq X} |\Lambda_F(p)|^2 + \sum_{\substack{p^k \leq X \\ k \geq 2}} |\Lambda_F(p^k)|^2 .
\end{equation*}
For $k = 1$, we have $\Lambda_F(p) = A_F(p, 1) \log p$. Hence, by partial summation again,
\begin{equation*}
\sum_{p \leq X} |\Lambda_F(p)|^2 \ll \log^2 X \sum_{p \leq X} |A_F(p, 1)|^2 \ll X \log^2 X.
\end{equation*}
For $k \geq 2$, applying the Kim--Sarnak bound gives $|\Lambda_F(p^k)| \ll p^{\frac{5k}{14}} \log p$. The contribution from squares ($k = 2$) and powers $k \geq 3$ satisfies
\begin{equation*}
\sum_{p^2 \leq X} |\Lambda_F(p^2)|^2 + \sum_{k \geq 3} \sum_{p^k \leq X} |\Lambda_F(p^k)|^2 \ll X^{6/7} \log^2 X + X^{27/28+ \varepsilon} \ll X^{27/28+ \varepsilon} .
\end{equation*}
Because $X^{27/28+ \varepsilon} = o(X \log X)$, the prime contribution dominates, yielding \eqref{eq:lem-LambdaF}.

By Cauchy--Schwarz on the convolution definition of $\psi_F(m)$, we have
\begin{equation*}
|\psi_F(m)|^2 \leq d(m) \sum_{m = m_1 m_2} |\Lambda_F(m_1)|^2 |\mu_F(m_2)|^2 .
\end{equation*}
Summing over $m \leq X$ gives
\begin{align*}
\sum_{m \leq X} |\psi_F(m)|^2 &\ll \sum_{m_1 \leq X} d(m_1) |\Lambda_F(m_1)|^2 \sum_{m_2 \leq \frac{X}{m_1}} d(m_2) |\mu_F(m_2)|^2\\
&\ll X \log X \sum_{m_1 \leq X} \frac{d(m_1) |\Lambda_F(m_1)|^2}{m_1} \\
&\ll X \log^3 X.
\end{align*}
The definition $\varphi_F(l)$ and Cauchy--Schwarz implies
\begin{equation*}
|\varphi_F(l)|^2 \leq d(l) \sum_{l = l_1 l_2} |\mu_F(l_1)|^2 |A_F(l_2, 1)|^2.
\end{equation*}
Summing over $l \leq X$ gives
\begin{align*}
\sum_{l \leq X} |\varphi_F(l)|^2 &\ll  \sum_{l_1 \leq X} d(l_1) |\mu_F(l_1)|^2 \sum_{l_2 \leq \frac{X}{l_1}} d(l_2) |A_F(l_2, 1)|^2 \\
&\ll X \log X \sum_{l_1 \leq X} \frac{d(l_1) |\mu_F(l_1)|^2}{l_1} \\
&\ll  X \log^3 X. 
\end{align*}
\end{proof}

\subsection{Proof of Theorem \ref{thm:linear}}
Multiplying \eqref{eq:split into Type I II sums} by $e(n\alpha)$ and adding up to $n\leq X$ gives
\begin{equation*}
S(F, X) \coloneqq \sum_{n \leq X} \Lambda_F(n) e(n\alpha) = S_1 + S_2 + S_3 + S_4 + O\bigl((UV)^{1+\varepsilon}\bigr),
\end{equation*}
where
\begin{align*}
S_1 &= \sum_{n \leq U} \Lambda_F(n) e(n\alpha), \\
S_2 &= \sum_{m \leq V} \mu_F(m) \sum_{l \leq X/m} A_F(l, 1) \log l \, e(ml\alpha), \\
S_3 &= \sum_{m \leq UV} \psi_F(m) \sum_{l \leq X/m} A_F(l, 1) e(ml\alpha), \\
S_4 &= \sum_{\substack{m > U, \, l > V \\ ml \leq X}} \Lambda_F(m) \varphi_F(l) e(ml\alpha).
\end{align*}
We know that
\begin{equation*}
\sum_{n \leq X} A_F(n, 1) e(n\alpha) \ll_{F, \varepsilon} X^{1/2+\rho + \varepsilon},
\end{equation*}
holds with $\rho=1/4$, and by partial summation, we obtain
\begin{equation*}
\sum_{n \leq X} A_F(n, 1) e(n\alpha) \log n \ll_{F, \varepsilon} X^{1/2+\rho + \varepsilon}.
\end{equation*}

Now, we estimate each sum
\begin{align*}
S_1 &\ll U \log^2U, \\
S_2 &\ll \sum_{m \leq V} |\mu_F(m)| \left( \frac{X}{m} \right)^{1/2+\rho + \varepsilon} 
\ll X^{1/2+\rho + \varepsilon} \Biggl( \sum_{m \leq V} |\mu_F(m)|^2 \Biggr)^{1/2} 
\ll X^{1/2+\rho + \varepsilon} V^{1/2} (\log V)^{1/2}, \\
S_3 &\ll \sum_{m \leq UV} |\psi_F(m)| \left( \frac{X}{m} \right)^{1/2+\rho + \varepsilon} 
\ll X^{1/2+\rho + \varepsilon} \Biggl( \sum_{m \leq UV} |\psi_F(m)|^2 \Biggr)^{1/2} \\
&\ll X^{1/2+\rho + \varepsilon} (UV)^{1/2} (\log UV)^{1/2}.
\end{align*}
For $S_4$, by dyadic decomposition, we can split it into $\ll \log^2 X$ sums of the form
\begin{equation*}
\sum_{m \sim M} \sum_{l \sim L} \varphi_F(l) \Lambda_F(m) e(ml\alpha),
\end{equation*}
where $U < M \ll X/V$, $V < L \ll X/U$, $UV < X$, and $ML \sim X$.

By standard estimates for bilinear sums, we obtain
\begin{equation*}
S_4 \ll X^{1+\varepsilon} \left( \frac{1}{M} + \frac{1}{L} + \frac{q}{X} + \frac{1}{q} \right)^{1/4} .
\end{equation*}
Thus, combining all bounds gives
\begin{align*}
S(F, X) &\ll  (UV)^{1+\varepsilon}+ X^{1/2+\rho + \varepsilon} (UV)^{1/2 + \varepsilon} + X^{1 + \varepsilon} \left( \frac{1}{U} + \frac{1}{V} + \frac{q}{X} + \frac{1}{q} \right)^{1/4}.
\end{align*}
Since $\alpha$ is of type $1$, for any $Q$, we can choose $q$ such that $Q^{1 - \varepsilon} < q \leq Q$. Then
\begin{equation*}
S(F, X) \ll (UV)^{1+\varepsilon} + X^{1/2+\rho + \varepsilon} (UV)^{1/2 + \varepsilon} + X^{1 + \varepsilon} \left( \frac{1}{U} + \frac{1}{V} + \frac{Q}{X} + \frac{1}{Q} \right)^{1/4}.
\end{equation*}
We choose $Q = X^{1/2}$ and $U = V$. Then
\begin{equation*}
S(F, X) \ll U^{2+\varepsilon} + X^{1/2+\rho + \varepsilon} U^{1 + \varepsilon} + X^{1 + \varepsilon} \left( \frac{1}{U} \right)^{1/4}.
\end{equation*}
Balancing $x^{1/2+\rho } U = X U^{-1/4}$, and choosing $U = X^{2/5-4\rho/5}$, implies
\begin{equation*}
S(F, X) \ll  X^{1 - 1/10+\rho/5 + \varepsilon} \ll X^{\frac{9+2\rho}{10} + \varepsilon}.
\end{equation*}
This completes the proof of Theorem \ref{thm:linear}.

\section{Non-linear case}
\subsection{Preliminaries}
\begin{lemma}[Exponent Pair Lemma]\label{lem:exponent_pair}
Let $g$ be a smooth, real-valued function defined on an interval containing $[B, B+h]$ with $0 < h \leq B$. Suppose there exists a parameter $A = A(g) > 0$ such that for all integers $j \geq 1$ and all $x \asymp B$,
\begin{equation*}\label{eq:deriv_bounds}
\left| g^{(j)}(x) \right| \asymp_j A B^{1-j} ,
\end{equation*}
where the implied constants depend solely on $j$. Then for any exponent pair $(\kappa, \lambda)$ satisfying
\begin{equation*}
0 \leq \kappa \leq \frac{1}{2} \leq \lambda \leq 1 ,
\end{equation*}
we have the exponential sum estimate
\begin{equation*}
\sum_{B < n \leq B+h} e(g(n)) \ll A^\kappa B^\lambda. 
\end{equation*}
\end{lemma}

\begin{proof}
See Graham and Kolesnik \cite{GK1991}. Also see, \cite[Section 2.3]{Ivic}
\end{proof}

\begin{lemma}[Weyl--van der Corput Inequality]
Let $\{\xi(m)\}$ be a complex-valued sequence supported on an interval $I$ of length $M = |I|$. If $H$ is any positive integer, then
\begin{equation}
\label{eq:vdc1}
\left| \sum_{m} \xi(m) \right|^2 \leq \frac{|I| + H}{H} \sum_{|h| < H} \left(1 - \frac{|h|}{H}\right) \sum_{m} \xi(m) \overline{\xi(m-h)}.
\end{equation}
Separating the diagonal term ($h = 0$) and grouping non-zero symmetric shifts ($h > 0$ and $h < 0$), inequality \eqref{eq:vdc1} implies
\begin{equation}
\label{eq:vdc2}
\left| \sum_{m} \xi(m) \right|^2 \leq \frac{|I| + H}{H} \sum_{m} |\xi(m)|^2 + \frac{2(|I| + H)}{H} \sum_{1 \leq h < H} \left(1 - \frac{h}{H}\right) \operatorname{Re} \sum_{m} \xi(m+h) \overline{\xi(m)}.
\end{equation}
\label{Lemma:Weyl--van der Corput inequality}
\end{lemma}

\begin{proof}
Identity \eqref{eq:vdc1} is established by standard smoothing over shifted intervals (see e.g.\ Graham and Kolesnik \cite[Lemma~2.5]{GK1991}). For $h=0$, we get the first term in \eqref{eq:vdc2}. 
For $h \neq 0$, pairing $+h$ and $-h$ gives
\[
\sum_{m} \xi(m) \overline{\xi(m-h)} + \sum_{m} \xi(m) \overline{\xi(m+h)}.
\]
Re-indexing the first sum with $m' = m - h$ yields $\sum_{m'} \xi(m'+h) \overline{\xi(m')}$, while the second sum is its complex conjugate
\[
\sum_{m} \xi(m) \overline{\xi(m-h)} + \sum_{m} \xi(m) \overline{\xi(m+h)} = 2 \operatorname{Re} \sum_{m} \xi(m+h) \overline{\xi(m)}.
\]
\end{proof}

As an application of Cauchy--Schwarz and the Weyl--van der Corput inequality, we derive the following lemma.

\begin{lemma} 
Let $M, N \geq 2$ and $X \asymp MN$. Let $\{a(m): m \sim M\}$ and $\{b(n): n \sim N\}$ be two arbitrary sequences of complex numbers satisfying the mean-value bounds
\begin{equation*}
\sum_{m \sim M} |a(m)|^2 \ll M \log^{c_0} M \qquad \text{and} \qquad \sum_{n \sim N} |b(n)|^2 \ll N \log^{c_0} N
\end{equation*}
for some constant $c_0 > 0$. Consider the bilinear form
\begin{equation*}
S = \sum_{m \sim M} \sum_{n \sim N} a(m) b(n) e\left(f(m,n)\right),
\end{equation*}
where $f(m,n) = \alpha m^\theta n^\theta$ with $0 < \theta < 1$. Then, we have
\begin{equation*} 
|S| \ll M N^{\frac{165}{194}} X^{\frac{13(\theta-1)}{194}}.
\end{equation*}
\end{lemma}
\begin{proof}
Applying the Cauchy--Schwarz inequality across the variable $n$ gives
\[
|S|^2 \leq \left(\sum_{n \sim N} |b(n)|^2\right) \sum_{n \sim N} \left| \sum_{m \sim M} a(m) e\left(f(m,n)\right) \right|^2.
\]
Setting $\xi(m) = a(m) e(f(m,n))$ in Lemma~\ref{Lemma:Weyl--van der Corput inequality} with $1 \leq H \leq M$, we obtain
\begin{align*}
|S|^2 &\ll N^{1+\varepsilon} \sum_{n \sim N} \left[ \frac{M}{H} \sum_{m \sim M} |a(m)|^2 + \frac{M}{H} \sum_{1 \leq h<H}  \sum_{m \sim M} a(m+h)\overline{a(m)} e\left(\Delta_h f(m,n)\right)  \right] \\
&\ll \frac{(MN)^{2+\varepsilon}}{H} + \frac{MN}{H} \sum_{1 \leq h<H} \sum_{m \sim M} |a(m+h) a(m)| \left| \sum_{n \sim N} e\left(\Delta_h f(m,n)\right) \right|,
\end{align*}
where $\Delta_h f(m,n) = f(m+h,n) - f(m,n)$.\\
By the Mean Value Theorem with respect to $m$
\[
(m+h)^\theta - m^\theta = \theta h m^{\theta-1} + O\left(h^2 M^{\theta-2}\right).
\]
For fixed $h$ and $m$, write $g(n) := \Delta_h f(m,n) \approx \alpha \theta h m^{\theta-1} n^\theta$. Differentiating with respect to $n$ gives
\[
|g'(n)| = \alpha \theta^2 h m^{\theta-1} n^{\theta-1} = \alpha \theta^2 h (mn)^{\theta-1} \asymp \alpha h X^{\theta-1}.
\]
Applying an exponent pair $(k, \lambda)$ with $0 \leq k \leq \frac{1}{2} \leq \lambda \leq 1$ to the inner $n$-sum gives
\[
\left| \sum_{n \sim N} e\left(\Delta_h f(m,n)\right) \right| \ll \left(\alpha h X^{\theta-1}\right)^k N^\lambda.
\]
Summing over $m \sim M$ (which yields $\ll M$) and over $1 \leq h<H$ (which yields $H^{k+1}$), the off-diagonal contribution becomes
\[
\frac{MN}{H} \cdot M \cdot H^{k+1} X^{k(\theta-1)} N^\lambda = M^2 N^{\lambda+1} H^k X^{k(\theta-1)}.
\]
The full bound on $|S|^2$ is therefore
\begin{equation*}
|S|^2 \ll \frac{(MN)^{2+\varepsilon}}{H} + M^2 N^{\lambda+1} H^k X^{k(\theta-1)}.
\end{equation*}
Equating the diagonal and off-diagonal terms gives $H = N^{\frac{1-\lambda}{k+1}} X^{\frac{k(1-\theta)}{k+1}}.$
Substituting $H$ back yields
\begin{equation*}
\label{eq:general_S_bound}
|S| \ll M N^{\frac{2k+1+\lambda}{2(k+1)}} X^{\frac{k(\theta-1)}{2(k+1)}}.
\end{equation*}
Bourgain \cite{Bourgain2017} established the exponent pair $(k, \lambda) = \left(\frac{13}{84}, \frac{55}{84} \right).$ Notice the exact algebraic identity $\lambda - k = \frac{1}{2}$. Consequently,
\[
\frac{2k+1+\lambda}{2(k+1)} = \frac{165}{194}.
\]
Therefore, the optimal shift parameter $H$ and the bilinear bound $|S|$ are
\begin{equation*}
H = N^{\frac{29}{97}} X^{\frac{13(1-\theta)}{97}}, \qquad |S| \ll M N^{\frac{165}{194}} X^{\frac{13(\theta-1)}{194}}.
\end{equation*}

\end{proof}

\subsection{Proof of Theorem \ref{thm:non-linear}}

Applying Vaughan's identity to $S(F, X, \theta) $ decomposes the sum into
\[
S(F, X, \theta) = S_1 + S_2 + S_3 + S_4 + O\left((UV)^{1+\varepsilon}\right)
\]
Here we set $U = X^u$ and $V = X^v$, where $u, v \geq 0$ are open parameters, which we choose later. The Type~I and remainder terms satisfy
\begin{align*}
S_1 &\ll U, \\
S_2 &\ll X^{\frac{3}{4} + \frac{9\theta}{28}} V^{1/2} \log^{1/2} V, \\
S_3 &\ll X^{\frac{3}{4} + \frac{9\theta}{28}} (UV)^{1/2} \log^{1/2}(UV).
\end{align*}
The Type~II contribution $S_4$ decomposes into $\ll \log^2 X$ dyadic bilinear sums $S$ with $U < M \leq X/V$, $V < N \leq X/U$, $  UV< X$, and $MN \asymp X$. Using our bound with $M \asymp X/N$
\[
S_4 \ll \left(\frac{X}{N}\right) N^{\frac{165}{194}} X^{\frac{13(\theta-1)}{194}} = X^{1 + \frac{13(\theta-1)}{194}} N^{-\frac{29}{194}} = X^{\frac{181+13\theta}{194}} N^{-\frac{29}{194}}.
\]
Because $-\frac{29}{194} < 0$ and $N \geq V = X^v$, the supremum is attained at $N = X^v$, so the overall bound is
\[
S(F, X, \theta) \ll X^{u+v} + X^{\frac{3}{4} + \frac{9\theta}{28} + \frac{u+v}{2}} + X^{\frac{181+13\theta-29v}{194}}.
\]
Since all terms are monotonically increasing in $u$, we choose $u \to 0$ ($U \approx 1$), reducing the estimate to
\begin{equation*}
S(F, X, \theta) \ll X^v + X^{\frac{3}{4} + \frac{9\theta}{28} + \frac{v}{2}} + X^{\frac{181+13\theta-29v}{194}}.
\end{equation*}
Equating the last two exponents to gives $v = \frac{497 - 691\theta}{1764}$, and since $v > 0$, implies $\theta < \frac{497}{691} \approx 0.71925.$
Substituting this optimal $v$ back into the exponent gives
\begin{equation*}
S(F, X, \theta) \ll X^{\frac{3143 + 443\theta}{3528} + \varepsilon} \approx X^{0.89087 + 0.12557\,\theta + \varepsilon}.
\end{equation*}

The non-triviality condition $\frac{3143 + 443\theta}{3528} < 1 $ requires $ \theta < \frac{385}{443} \approx 0.86907$, 
which holds throughout the domain $\theta \in \left(0, \frac{497}{691}\right)$.

\subsection{Verification of the Shift Parameter Range $1 \leq H \leq M$}

The optimal shift parameter is $H = N^{\frac{29}{97}} X^{\frac{13(1-\theta)}{97}}.$ Setting $N = X^n$ and $M = X^m$ with $m + n = 1$ ($n, m \geq 0$), the exponent of $H$ in powers of $X$ is $ \frac{29n + 13(1-\theta)}{97}$. Since $0<\theta<1$, clearly the exponent is positive. Hence, $H \geq 1$ holds unconditionally. Now, comparing $H$ to $M = X^{1-n}$ gives $n \leq \frac{84 + 13\theta}{126}$, which is equivalent to $m \geq \frac{42 - 13\theta}{126}$.
At $\theta = 0$, this requires $m \geq \frac{42}{126} = \frac{1}{3} $. At the boundary $M \asymp X/V = X^{1-v}$, since $v = \frac{497 - 691\theta}{1764} \leq \frac{497}{1764}$, we have $m = 1 - v \geq 1 - 0.28175 = 0.71825 \gg \frac{1}{3}$. Therefore, $1 \leq H \leq M$ and $U \leq M \leq X/V$ are strictly satisfied in the entire valid range.

\subsection{Permissible Range of $\theta$ and the Balanced Exponents}\label{Sec-analysis-exp-pair}

Although our primary choice is Bourgain's exponent pair $(k, \lambda) = \left(\frac{13}{84}, \frac{55}{84}\right)$, the derivation applies equally to other standard exponent pairs. For any valid pair $(k, \lambda)$, the condition $v > 0$ determines the admissible range $\theta < \theta_{\max}(k)$, while balancing the Type~I and Type~II exponents yields the corresponding estimate $S(F, X, \theta) \ll X^{E_{\text{bal}}(\theta)+\varepsilon}$. 

A comparison across different choices of $(k, \lambda)$ illustrates how decreasing $k$ systematically extends the admissible domain of $\theta$
\begin{itemize}
    \item \textbf{Classical pair} $(k, \lambda) = \left(\frac{1}{2}, \frac{1}{2}\right)$:
    \[
    \theta < \frac{7}{13} \approx 0.53846, \qquad E_{\text{bal}}(\theta) = \frac{13}{16} + \frac{23\theta}{112} = \frac{91+23\theta}{112}.
    \]
    
    \item \textbf{$A$-process pair} $(k, \lambda) = A\left(\frac{1}{2}, \frac{1}{2}\right) = \left(\frac{1}{6}, \frac{2}{3}\right)$:
    \[
    \theta < \frac{5}{7} \approx 0.71429, \qquad E_{\text{bal}}(\theta) = \frac{8}{9} + \frac{8\theta}{63} = \frac{56+8\theta}{63}.
    \]
    
    \item \textbf{Rankin's pair} $(k, \lambda) = \left(\frac{1}{2}\Theta + \varepsilon, \frac{1}{2} + \frac{1}{2}\Theta + \varepsilon\right)$ with $\Theta \approx 0.329021$:
    \[
    \theta < \frac{7(2-\Theta)}{18-5\Theta} \approx 0.71519, \qquad E_{\text{bal}}(\theta) = \frac{11-\Theta}{12} + \theta\left(\frac{9+5\Theta}{84}\right).
    \]
    
    \item \textbf{Bourgain's pair} $(k, \lambda) = \left(\frac{13}{84} , \frac{55}{84} \right)$:
    \[
    \mathbf{\theta < \frac{497}{691} \approx 0.71925}, \qquad E_{\text{bal}}(\theta) = \frac{449}{504} + \frac{443\theta}{3528} = \frac{3143+443\theta}{3528}.
    \]
    
    \item \textbf{Conjectured pair} $(k, \lambda) = \left(\varepsilon, \frac{1}{2} + \varepsilon\right)$ (under the Exponent Pairs Conjecture):
    \[
    \theta < \frac{7}{9} \approx 0.77778, \qquad E_{\text{bal}}(\theta) = \frac{11}{12} + \frac{3\theta}{28} = \frac{77+9\theta}{84} .
    \]
\end{itemize}
In particular, Bourgain's pair provides the widest unconditional range of non-triviality ($\theta < \frac{497}{691}$) among the standard symmetric exponent pairs satisfying $\lambda - k = \frac{1}{2}$, while the limit $\theta < \frac{7}{9}$ represents the natural barrier of the present decomposition.

\subsection*{Acknowledgements.} The author gratefully acknowledges the financial support from the National Board for Higher Mathematics (NBHM), India, through its postdoctoral fellowship (Ref. No. 0204/9(1)/2026-R$\And$D-II/4788). The author expresses sincere gratitude to Prof. Ritabrata Munshi for his valuable suggestions throughout this work.

\end{document}